\documentclass[a4paper]{amsart}
\usepackage{latexsym}\usepackage{ifthen}\usepackage[leqno]{amsmath}
\usepackage{enumerate}\usepackage{calc}\usepackage{hyphenat}
\usepackage{amstext,amsbsy,amsopn,amsthm,amsgen}
\usepackage{amsfonts,amscd,amsxtra,upref}
\usepackage{mathrsfs}\usepackage{euscript}\usepackage{amssymb}
\swapnumbers
\theoremstyle{plain}
\makeatletter\@namedef{subjclassname@2010}{\textup{2010} Mathematics Subject Classification}
\makeatother
\newtheorem{Thm}{Theorem}[section]
\newtheorem{Lem}[Thm]{Lemma}
\newtheorem{Cor}[Thm]{Corollary}
\newtheorem{Pro}[Thm]{Proposition}

\theoremstyle{definition}
\newtheorem{Def}[Thm]{Definition}

\theoremstyle{remark}
\newtheorem{Rem}[Thm]{Remark}
\numberwithin{equation}{section}

\newcommand{\ITE}[3]{\ifthenelse{#1}{#2}{#3}}\newcommand{\ITEE}[4][]{\ITE{\equal{#2}{#3}}{#4}{#1}}
 \ITE{\isundefined{\texorpdfstring}}{\newcommand{\texorpdfstring}[2]{#1}}{}

\newcommand{\myData}[1][]{
 \author[P.\ Niemiec]{Piotr Niemiec}
 \address{\ITEE{#1}{*}{P.\ Niemiec{}\\}Instytut Matematyki\\
  Wydzia\l{} Matematyki i~Informatyki\\Uniwersytet Jagiello\'{n}ski\\
  ul.\ \L{}ojasiewicza 6\\30-348 Krak\'{o}w\\Poland}
 \email{piotr.niemiec@uj.edu.pl}
 }

\newenvironment{cor}[2][]{\ITEE[{\begin{Cor}[#1]}]{#1}{}{\begin{Cor}}\label{cor:#2}}{\end{Cor}}
\newenvironment{dfn}[2][]{\ITEE[{\begin{Def}[#1]}]{#1}{}{\begin{Def}}\label{def:#2}}{\end{Def}}
\newenvironment{lem}[2][]{\ITEE[{\begin{Lem}[#1]}]{#1}{}{\begin{Lem}}\label{lem:#2}}{\end{Lem}}
\newenvironment{pro}[2][]{\ITEE[{\begin{Pro}[#1]}]{#1}{}{\begin{Pro}}\label{pro:#2}}{\end{Pro}}
\newenvironment{rem}[2][]{\ITEE[{\begin{Rem}[#1]}]{#1}{}{\begin{Rem}}\label{rem:#2}}{\end{Rem}}
\newenvironment{thm}[2][]{\ITEE[{\begin{Thm}[#1]}]{#1}{}{\begin{Thm}}\label{thm:#2}}{\end{Thm}}

\newcommand{\COR}[2][!]{\ITEE{#1}{!}{Corollary~}\ITEE{#1}{s}{Corollaries~}\textup{\ref{cor:#2}}}
\newcommand{\LEM}[2][!]{\ITEE{#1}{!}{Lemma~}\ITEE{#1}{s}{Lemmas~}\textup{\ref{lem:#2}}}
\newcommand{\THM}[2][!]{\ITEE{#1}{!}{Theorem~}\ITEE{#1}{s}{Theorems~}\textup{\ref{thm:#2}}}

\newcommand{\RRR}{\mathbb{R}}
\newcommand{\SSS}{\mathbb{S}}

\newcommand{\dd}{\colon}
\newcommand{\epsi}{\varepsilon}
\newcommand{\geqsl}{\geqslant}
\newcommand{\leqsl}{\leqslant}
\newcommand{\scalar}[2]{\left\langle#1,#2\right\rangle}

\newcommand{\Iso}{\operatorname{Iso}}
\newcommand{\stab}{\operatorname{stab}}

\newcommand{\tfcae}{the following conditions are equivalent:}

\begin{document}

\title[Hyperbolic geometry for non-differential topologists]
 {Hyperbolic geometry\\for non-differential topologists}
\myData[*]
\author[P. Pikul]{Piotr Pikul}
\address{P. Pikul\\Instytut Matematyki\\
 Wydzia\l{} Matematyki i~Informatyki\\Uniwersytet Jagiello\'{n}ski\\
 ul.\ \L{}ojasiewicza 6\\30-348 Krak\'{o}w\\Poland}
\email{piotr.pikul@student.uj.edu.pl}
\begin{abstract}
A \textsl{soft} presentation of hyperbolic spaces, free of differential
apparatus, is offered. Fifth Euclid's postulate in such spaces is overthrown
and, among other things, it is proved that spheres (equipped with great-circle
distances) and hyperbolic and Euclidean spaces are the only locally compact
geodesic (i.e., convex) metric spaces that are three-point homogeneous.
\end{abstract}
\maketitle


\section{Introduction}

Hyperbolic geometry of a plane is known as a historically first example of
a non-Euclidean geometry; that is, geometry in which all Euclid's postulates are
satisfied, apart from the fifth, called \textsl{parallel}, which is false.
Discovered in the first half of the 19th century, is one of the greatest
mathematical achievements of those times. Although it merits special attention
and mathematicians all over the world hear about it sooner or later, there are
plenty of them whose knowledge on hyperbolic geometry is greedy and far from
formal details. This sad truth concerns also topologists which do not specialise
in differential geometry. One of reasons for this state is concerned with
the extent of differential apparatus that one needs to learn in order to get to
know and understand hyperbolic spaces. It was the main sake for us to propose
and prepare an introduction to hyperbolic spaces (and geometry) that will be
self-contained, free of differential language and accessible to `everyone.'\par
The other story about hyperbolic spaces concerns their \textsl{free mobility}
(or, in other words, \textsl{absolute} metric \textsl{homogeneity}). A metric
space is absolutely homogeneous if all its partial isometries (that is,
isometries between its subspaces) extend to global (bijective) isometries.
According to a deep result from the 50's of the 20th century, known almost only
by differential/Riemannian geometrists, hyperbolic spaces, beside Euclidean
spaces and Euclidean spheres, are (in a very strong sense) the only connected
locally compact metric spaces that have this property. One may even assume less
about a connected locally compact metric space---that only partial isometries
between $3$-point and $2$-point subspaces extend to global isometries. Then such
a space is `equivalent' to one of the aforementioned Riemannian manifolds and
therefore is automatically absolutely homogeneous (the equivalence we speak here
about does not imply that spaces are isometric, but is much stronger than
a statement that they are homeomorphic; full details about classification of all
such spaces up to isometry can be found in Section~5 below---see
\THM{classification}). So, high level of metric homogeneity makes hyperbolic
spaces highly exceptional. This is another reason for putting special attention
on them.\par
The paper is organised as follows. In Section~2 we introduce (one of possible
models of) hyperbolic spaces and prove that their metrics satisfy the triangle
inequality and are equivalent to Euclidean metrics. In the next, third, section
we show that hyperbolic spaces are absolutely homogeneous and admit no dilations
other than isometries (for dimension greater than $1$). Fourth section is
devoted to Euclid's postulates, where it is proved that parallel postulate is
false in hyperbolic spaces, whereas all other postulates are satisfied. To make
the presentation self-contained, we also define there straight lines and their
segments as well as angles, and show that one-dimensional hyperbolic space is
isometric to the ordinary real line. The last, fifth, section is devoted to
the classification (up to isometry) of all connected locally compact metric
spaces that are $3$-point homogeneous. All proofs are included, apart from
the proof of a deep and difficult result due to Freudenthal \cite{fr1} (see
\THM{freu}). This result is applied only once in this paper---to classify
$3$-point homogeneous metric spaces described above.\par
The reader interested in Riemannian geometry may consult, e.g., \cite{mor} or
\cite{cha}.

\subsection*{Notation and terminology}
In this paper the term \textsl{metric} means a function (in two variables) that
assigns to a pair of points of a fixed set their distance (so, \textsl{metric}
does \underline{not} mean \textsl{Riemannian metric}, a common and important
notion in differential geometry).\par
For a pair of vectors $x = (x_1,\ldots,x_n)$ and $y = (y_1,\ldots,y_n)$
in $\RRR^n$, we denote by $\scalar{x}{y}$ their standard inner product (that is,
$\scalar{x}{y} = \sum_{k=1}^n x_k y_k$), whereas $\|x\| = \sqrt{\scalar{x}{x}}$
is the Euclidean norm of $x$. The Euclidean metric (induced by $\|\cdot\|$) will
be denoted as $d_e$. To simplify further arguments and statements, we introduce
the following notation:
\begin{equation}\label{eqn:bracket}
[x] = \sqrt{1+\|x\|^2}.
\end{equation}\par
By an \textit{isometric} map between metric spaces we mean any function that
preserves the distances, that is, $f\dd (X,d_X) \to (Y,d_Y)$ is isometric if
$d_Y(f(p),f(q)) = d_X(p,q)$ for any $p, q \in X$. The term \textit{isometry} is
reserved for surjective isometric maps. Additionally, the above map $f$ is
a \textit{dilation} if there is a constant $c > 0$ such that $d_Y(f(p),f(q)) =
c d_X(p,q)$ for any $p, q \in X$ (we do not assume that dilations are
surjective). We will denote by $\Iso(X,d_X)$ the full isometry group of
$(X,d_X)$; that is,
\[\Iso(X,d_X) = \{u\dd (X,d_X) \to (X,d_X),\quad u \textup{ isometry}\}.\]
A metric space $(X,d_X)$ is said to be \textit{geodesic} (or \textit{convex}) if
for any two distinct points $a$ and $b$ of $X$ there exists a dilation
$\gamma\dd [0,1] \to (X,d_X)$ such that $\gamma(0) = a$ and $\gamma(1) = b$ (we
do not assume the uniqueness of such $\gamma$).\par
For the reader's convenience, let us recall that the inverse hyperbolic cosine
is defined as $\cosh^{-1}(t) = \log(t + \sqrt{t^2-1})$ for $t \geqsl 1$ (in
whole this paper `$\log$' stands for the natural logarithm).

\section{Hyperbolic distance}

Below we introduce one of many equivalent models of hyperbolic spaces---the one
most convenient for us.

\begin{dfn}{hyper}
The \textit{$n$-dimensional} (real) \textit{hyperbolic space} is a metric space
\[(H^n(\RRR),d_h)\] where $H^n(\RRR) = \RRR^n$ and $d_h$ is a metric (called
\textit{hyperbolic}) given by
\begin{equation}\label{eqn:metric}
d_h(x,y) = \cosh^{-1}([x][y] - \scalar{x}{y}) \qquad (x,y \in \RRR^n)
\end{equation}
(see \eqref{eqn:bracket}).
\end{dfn}

The above formula has its origin in differential/Riemannian geometry (most often
it is defined as the length of a geodesic arc---so, to get it one needs to find
geodesics and compute integrals related to them), see, e.g., \cite{rey}. We
underline here---at the very beginning of our presentation---that establishing
the triangle inequality for $d_h$ using elementary methods is undoubtedly
the most difficult part in whole this approach.

\begin{rem}{hyper}
As we will see in \COR{n=1}, the metric space $(H^1(\RRR),d_h)$ is isometric to
$(\RRR,d_e)$. This contrasts with all other cases, as for $n > 1$ the metric
space $(H^n(\RRR),d_h)$ admits no dilations other than isometries (see
\THM{dilation} below). In particular, for any integer $n > 1$ and positive real
$r \neq 1$ the metric spaces $(H^n(\RRR),d_h)$ and $(H^n(\RRR),r d_h)$ are
\textsl{different} (i.e., they are non-isometric). Each of the metrics $r d_h$
(with fixed $r > 0$) may serve as a `standard' hyperbolic metric. Actually,
\textsl{everything} that will be proved in this paper about the metric spaces
$(H^n(\RRR),d_h)$ remains true when the metric is replaced by $r d_h$.\par
Note also that, similarly as practiced with Euclidean spaces, for $j < k$ the
space $H^j(\RRR)$ can naturally be considered as the subspace $\RRR^j \times
\{0\}^{k-j}$ of $H^k(\RRR)$. Under this identification, the hyperbolic metric
of $H^j(\RRR)$ coincides with the metric induced from the hyperbolic one of
$H^k(\RRR)$. This is the main reason why we `forget' the dimension $n$ in
the notations `$d_h$' and `$d_e$.'
\end{rem}

The aim of this section is to show that $d_h$ is a metric on $\RRR^n$ equivalent
to $d_e$.

\begin{lem}{well}
For any $x, y \in H^n(\RRR)$, $d_h(x,y)$ is well defined, non-negative and
$d_h(x,y) = d_h(y,x)$. Moreover, $d_h(x,y) = 0$ iff $x = y$.
\end{lem}
\begin{proof}
It follows from the Schwarz inequality that
\begin{equation}\label{Schwarz}
\|x\|^2 \|y\|^2 + \|x-y\|^2 \geqsl \scalar{x}{y}^2.
\end{equation}
Equivalently, $\|x\|^2 \|y\|^2 + \|x\|^2 - 2\scalar{x}{y} + \|y\|^2 \geqsl
\scalar{x}{y}^2$, which gives $1 + \|x\|^2 + \|y\|^2 + \|x\|^2 \|y\|^2 \geqsl
1 + 2\scalar{x}{y} + \scalar{x}{y}^2$ and hence $(1+\|x\|^2)(1+\|y\|^2) \geqsl
(1+\scalar{x}{y})^2$. Taking square roots from both sides shows that $[x][y]
- \scalar{x}{y} \geqsl 1$ and thus $d_h(x,y)$ is well defined (and, of course,
non-negative). Further, we see from \eqref{Schwarz} that $[x][y] - \scalar{x}{y}
= 1$ iff $x = y$, which gives the last claim of the lemma. Symmetry is trivial.
\end{proof}

To establish the triangle inequality, first we show its special case, which will
be used later in the proof of a general case.

\begin{lem}{tri_ineq}
For any $x, y \in H^n(\RRR)$, $d_h(x,y) \leqsl d_h(x,0) + d_h(0,y)$, and
equality appears iff either $x = ty$ or $y = tx$ for some $t \leqsl 0$.
\end{lem}
\begin{proof}
Observe that
\begin{multline*}
d_h(x,0) + d_h(0,y) = \cosh^{-1}([x]) + \cosh^{-1}([y])\\= \log([x]+\|x\|) +
\log([y]+\|y\|) = \log([x][y]+[x]\,\|y\|+[y]\,\|x\|+\|x\|\,\|y\|).
\end{multline*}
On the other hand,
\[d_h(x,y) = \cosh^{-1}([x][y]-\scalar{x}{y}) \leqsl
\cosh^{-1}([x][y]+\|x\|\,\|y\|)\]
and equality holds in the above iff $\scalar{x}{y} = -\|x\|\,\|y\|$, or,
equivalently, if either $x = ty$ or $y = tx$ for some $t \leqsl 0$. So,
to complete the whole proof, we only need to check that
\[[x][y]+[x]\,\|y\|+[y]\,\|x\|+\|x\|\,\|y\| =
\exp(\cosh^{-1}([x][y]+\|x\|\,\|y\|)),\]
which is left to the reader as an elementary exercise.
\end{proof}

To get the triangle inequality for general triples of elements of the hyperbolic
space, we will use certain one-dimensional perturbations of the identity map,
which turn out to be isometries (with respect to the hyperbolic distance). They
are introduced in the following

\begin{dfn}{translation}
For any $y \in H^n(\RRR)$ let a map $T_y\dd H^n(\RRR) \to H^n(\RRR)$ be defined
by \[T_y(x) = x + \Bigl([x]+\frac{\scalar{x}{y}}{[y]+1}\Bigr)y \qquad (x \in
H^n(\RRR)).\]
\end{dfn}

\begin{lem}{translation}
For any $y \in H^n(\RRR)$ the map $T_y$ is bijective and fulfills the equation:
\begin{equation}\label{eqn:trans-iso}
d_h(T_y(a),T_y(b)) = d_h(a,b) \qquad (a, b \in H^n(\RRR)).
\end{equation}
Moreover, $T_{-y}$ is the inverse of $T_y$.
\end{lem}
\begin{proof}
We start from computing $[T_y(u)]$ for $u \in H^n(\RRR)$:
\begin{multline*}
[T_y(u)]^2 = 1+\left\|u+\left([u]+\frac{\scalar{u}{y}}{[y]+1}\right)y\right\|^2
\\= 1+\|u\|^2+2\left([u]+\frac{\scalar{u}{y}}{[y]+1}\right)\scalar{u}{y}
+ \left([u]+\frac{\scalar{u}{y}}{[y]+1}\right)^2\|y\|^2\\
= [u]^2+\left(2[u]\scalar{u}{y}+2\frac{\scalar{u}{y}^2}{[y]+1}\right)
+ [u]^2\|y\|^2+\left(2[u]\scalar{u}{y}+\frac{\scalar{u}{y}^2}{[y]+1}\right)
\frac{[y]^2-1}{[y]+1}\\= [u]^2[y]^2+2[u]\scalar{u}{y}[y]+\scalar{u}{y}^2
\end{multline*}
and thus
\begin{equation}\label{eqn:T}
[T_y(u)] = [u]\,[y]+\scalar{u}{y}.
\end{equation}
For simplicity, put $\alpha(x) = [x]+\frac{\scalar{x}{y}}{[y]+1}$ (then $T_y(x)
= x + \alpha(x)y$). Observe that \eqref{eqn:trans-iso} is equivalent to
\[[T_y(a)]\,[T_y(b)]-\scalar{T_y(a)}{T_y(b)} = [a]\,[b]-\scalar{a}{b},\]
which can easily be transformed to an equivalent form:
\[[T_y(a)]\,[T_y(b)] = [a]\,[b]+\alpha(b)\scalar{a}{y}+\alpha(a)\scalar{b}{y}
+\alpha(a)\alpha(b)\|y\|^2.\]
The right-hand side expression of the above equation can be transformed as
follows:
\begin{multline*}
[a]\,[b]+\alpha(b)\scalar{a}{y}+\alpha(a)\scalar{b}{y}+\alpha(a)\alpha(b)\|y\|^2
\\=[a][b]+[b]\scalar{a}{y}+\frac{\scalar{b}{y}}{[y]+1}\scalar{a}{y}
+[a]\scalar{b}{y}+\frac{\scalar{a}{y}}{[y]+1}\scalar{b}{y}
+[a][b]([y]^2-1)\\+([b]\scalar{a}{y}+[a]\scalar{b}{y})\frac{[y]^2-1}{[y]+1}
+\frac{\scalar{a}{y}\scalar{b}{y}}{\left([y]+1\right)^2}([y]^2-1)\\
=[a][b][y]^2+[b]\scalar{a}{y}[y]+[a]\scalar{b}{y}[y]+\scalar{a}{y}\scalar{b}{y}
\\= ([a][y]+\scalar{a}{y})([b][y]+\scalar{b}{y}),
\end{multline*}
which equals $[T_y(a)]\,[T_y(b)]$, by \eqref{eqn:T}. So, \eqref{eqn:trans-iso}
is proved and, combined with \LEM{well}, implies that $T_y$ is one-to-one. To
finish the whole proof, it suffices to check that $T_{-y} \circ T_y$ coincides
with the identity map on $H^n(\RRR)$ (because then, by symmetry, also $T_y \circ
T_{-y}$ will coincide with the identity map). To this end, we fix $x \in
H^n(\RRR)$, and may and do assume that $y \neq 0$. Then there exist a unique
vector $z$ and a unique real number $\beta$ such that $\scalar{z}{y} = 0$ and
$x = z + \beta y$. Note that $T_y(x) = z+([x]+\beta[y])y$ and, consequently,
\begin{multline*}
T_{-y}(T_y(x)) = z+([x]+\beta[y])y-\Bigl([T_y(x)]-([x]+\beta[y])([y]-1)\Bigr)y\\
= z-[T_y(x)]y+([x]+\beta[y])[y]y.
\end{multline*}
Finally, an application of \eqref{eqn:T} enables us continuing the above
calculations to obtain:
\[T_{-y}(T_y(x)) = z-([x]\,[y]+\scalar{x}{y})y+([x]\,[y]+\beta[y]^2)y
= z-\beta(\|y\|^2-[y]^2)y = x.\]
\end{proof}

We are now able to prove the main result of this section.

\begin{thm}{metric}
The function $d_h$ is a metric equivalent to $d_e$.
\end{thm}
\begin{proof}
To show the triangle inequality, consider arbitrary three points $x$, $y$ and
$z$ of $H^n(\RRR)$. Since $T_y(0) = y$, it follows from \LEM[s]{translation}
and \ref{lem:tri_ineq} that
\[d_h(x,z) = d_h(T_{-y}(x),T_{-y}(z)) \leqsl d_h(T_{-y}(x),0)+d_h(0,T_{-y}(z))
= d_h(x,y)+d_h(y,z).\]
Further, to establish the equivalence of the metrics, observe that
\begin{equation}\label{eqn:dist0}
d_h(x,0) = \cosh^{-1}([x])
\end{equation}
and that both $T_y$ and $T_{-y}$ are continuous with respect to the Euclidean
metric (and thus they are homeomorphisms in this metric). Thus, for an arbitrary
sequence $(x_n)_{n=1}^{\infty}$ of elements of $\RRR^n$ and any $x \in \RRR^n$
we have (note that $T_{-x}(x) = 0$):
\begin{multline*}
x_n \stackrel{d_h}{\to} x \iff d_h(x_n,x) \to 0 \iff d_h(T_{-x}(x_n),0) \to 0
\iff [T_{-x}(x_n)] \to 1\\\iff T_{-x}(x_n) \stackrel{d_e}{\to} 0 \iff x_n
\stackrel{d_e}{\to} T_x(0) = x.
\end{multline*}
\end{proof}

The following is an immediate consequence of \THM{metric} (and the fact that
the collections of all closed balls around $0$ with respect to $d_h$ and $d_e$,
respectively, coincide---only radii change when switching between $d_h$ and
$d_e$). We skip its simple proof.

\begin{cor}{loc}
For each $n$, the hyperbolic space $H^n(\RRR)$ is locally compact and connected
and the metric $d_h$ is proper (that is, all closed balls are compact) and, in
particular, complete.
\end{cor}

\section{Absolute (metric) homogeneity}

\begin{dfn}{homogeneity}
A metric space $(X,d)$ is said to be \textit{absolutely} (\textit{metrically})
\textit{homogeneous} if any isometric map $f_0\dd (X_0,d) \to (X,d)$ defined on
a non-empty subset $X_0$ of $X$ extends to an isometry $f\dd (X,d) \to (X,d)$.
\end{dfn}

In this section we will show that hyperbolic spaces are absolutely homogeneous.
According to \THM{classification} (see Section~5 below), this property makes
them highly exceptional among all connected locally compact metric spaces.\par
The following is a reformulation of \LEM{translation}.

\begin{cor}{translation}
For any $y \in H^n(\RRR)$, the map $T_y\dd (H^n(\RRR),d_h) \to (H^n(\RRR),d_h)$
is an isometry.
\end{cor}

\begin{lem}{partial}
For a map $u\dd A \to H^n(\RRR)$ where $A$ is a subset of $H^n(\RRR)$ containing
the zero vector \tfcae
\begin{enumerate}[\upshape(i)]
\item $u$ is isometric (with respect to $d_h$) and $u(0) = 0$;
\item $\scalar{u(x)}{u(y)} = \scalar{x}{y}$ for all $x, y \in A$.
\end{enumerate}
\end{lem}
\begin{proof}
Assume (i) holds and fix $x, y \in A$. Then $[u(x)] = \cosh(d_h(u(x),u(0))) =
\cosh(d_h(x,0)) = [x]$. Similarly, $[u(y)] = [y]$ and thus $\scalar{u(x)}{u(y)}
= [u(x)]\,[u(y)]-\cosh(d_h(u(x),u(y))) = [x]\,[y]-\cosh(d_h(x,y)) =
\scalar{x}{y}$. Conversely, if (ii) holds and $x, y \in A$, then
$\scalar{u(x)}{u(x)} = \scalar{x}{x}$ and hence $[u(x)] = [x]$ (and, similarly,
$[u(y)] = [y]$) and $u(0) = 0$. But then easily $d_h(u(x),u(y)) = d_h(x,y)$ and
we are done. 
\end{proof}

\begin{thm}{homogeneous}
For any $n$ the hyperbolic space $(H^n(\RRR),d_h)$ is absolutely homogeneous.
\end{thm}
\begin{proof}
Fix an isometric map $v\dd A \to H^n(\RRR)$ where $A$ is a non-empty subset of
$H^n(\RRR)$. Take any $a \in A$, put $B = T_{-a}(A)$ and $u = T_{-v(a)} \circ
v \circ T_a\dd B \to H^n(\RRR)$, and observe that $0 \in B$, $u(0) = 0$ and $u$
is isometric (with respect to $d_h$). So, the map $u$ satisfies condition (ii)
from \LEM{partial}. It is well known (and easy to show) that each such a map
extends to a linear map $U\dd \RRR^n \to \RRR^n$ that corresponds (in
the canonical basis of $\RRR^n$) to an orthogonal matrix. The last property
means precisely that also $U$ fulfills the equation from condition (ii) of
\LEM{partial}. Hence, $U$ is isometric with respect to $d_h$. Then $T_{v(a)}
\circ U \circ T_{-a}$ is an isometry that extends $v$.
\end{proof}

The above proof enables to describe all isometries of the hyperbolic space
$H^n(\RRR)$: all of them are of the form $u = T_a \circ U$ where $a$ is a vector
and $U\dd \RRR^n \to \RRR^n$ is an orthogonal linear map (both $a$ and $U$ are
uniquely determined by $u$). In particular, the isotropy groups $\stab(x) =
\{u \in \Iso(H^n(\RRR),d_h)\dd\ u(x) = x\}$ of elements $x \in H^n(\RRR)$ are
pairwise isomorphic and $\stab(0)$ is precisely the group $O_n$ of all
orthogonal linear transformations of $\RRR^n$. The group $O_n$ also coincides
with the isotropy group of the zero vector with respect to the isometry group
of the $n$-dimensional Euclidean space. So, hyperbolic spaces are quite similar
to Euclidean. However, the last spaces have many (bijective) dilations, which
contrasts with hyperbolic geometry, as shown by

\begin{thm}{dilation}
For $n > 1 $, every dilation on $(H^n(\RRR),d_h)$ is an isometry.
\end{thm}
\begin{proof}
Assume $u\dd H^n(\RRR) \to H^n(\RRR)$ satisfies $d_h(u(x),u(y)) = c d_h(x,y)$
for all $x, y \in H^n(\RRR)$ and some constant $c > 0$. Our aim is to show that
then $c = 1$. (That is all we need to prove, since any global isometric map
on $H^n(\RRR)$ is onto---its inverse map is extendable to an isometry which
means that actually it is an isometry.)\par
Replacing $u$ by $T_{-u(0)} \circ u$, we may and do assume that $u(0) = 0$. Then
it follows from \eqref{eqn:dist0} that for any $x, y \in H^n(\RRR)$:
\begin{equation}\label{eqn:preserve}
\|f(x)\| = \|f(y)\| \iff \|x\| = \|y\|.
\end{equation}
Further, $d_h(x,0) = d_h(-x,0)$ and $d_h(x,-x) = d_h(x,0)+d_h(0,-x)$.
Consequently, $d_h(u(x),0) = d_h(u(-x),0)$ and $d_h(u(x),u(-x)) = d_h(u(x),0)+
d_h(0,u(-x))$. So, we infer from \LEM{tri_ineq} (and from \eqref{eqn:preserve})
that $u(x) = -u(x)$ for all $x \in H^n(\RRR)$.\par
Fix arbitrary two vectors $x, y \in H^n(\RRR)$ such that $\scalar{x}{y} = 0$.
Then \[\cosh(d_h(x,y)) = [x]\,[y] = [x]\,[-y] = \cosh(d_h(x,-y)),\] thus
$d_h(u(x),u(y)) = d_h(u(x),u(-y)) = d_h(u(x),-u(y))$. This implies that
\[\scalar{u(x)}{u(y)} = 0.\]\par
Now for arbitrary $t \geqsl 1$ choose two vectors $x, y\in H^n(\RRR)$ such that
$[x] = [y] = t$ and $\scalar{x}{y} = 0$. Then also $\scalar{u(x)}{u(y)} = 0$ and
(by \eqref{eqn:preserve}) $[u(y)] = [u(x)] = \cosh(d_h(u(x),0)) =
\cosh(c d_h(x,0)) = \cosh(c \cosh^{-1}(t))$. It follows from the former property
that $c \cosh^{-1}([x]\,[y]) = c d_h(x,y) = d_h(u(x),u(y)) = \cosh^{-1}([u(x)]\,
[u(y)])$, which combined with the latter yields
\begin{equation}\label{eqn:c=1}
\cosh(c \cosh^{-1}(t^2)) = (\cosh(c \cosh^{-1}(t)))^2.
\end{equation}
The above equation is valid for every $t \geqsl 1$ only if $c = 1$. Although
this is a well-known fact, for the reader's convenience we give its brief proof.
Observe that \[\cosh(c \cosh^{-1}(t^2)) =
\frac12\Bigl(\bigl(t^2+\sqrt{t^4-1}\bigr)^c+
\bigl(t^2+\sqrt{t^4-1}\bigr)^{-c}\Bigr)\] and \[(\cosh(c \cosh^{-1}(t)))^2 =
\frac14\Bigl(\bigl(t+\sqrt{t^2-1}\bigr)^c+
\bigl(t+\sqrt{t^2-1}\bigr)^{-c}\Bigr)^2.\] As a consequence,
$\lim_{t\to\infty} \frac{\cosh(c \cosh^{-1}(t^2))}{t^{2c}} = 2^c$, whereas
\[\lim_{t\to\infty} \frac{(\cosh(c \cosh^{-1}(t)))^2}{t^{2c}} = 2^{2c-1}.\] So,
if \eqref{eqn:c=1} holds for all $t \geqsl 1$, then $2^c = 2^{2c-1}$ and $c =
1$.
\end{proof}

We underline that the claim of \THM{dilation} is false for $n = 1$ (see
\COR{n=1} below).\par
As an immediate consequence of the above result we obtain

\begin{cor}{non-iso}
For any $n > 1$, the metric spaces $(H^n(\RRR),d_h)$ and $(\RRR^n,d_e)$ are
non-isometric.
\end{cor}

\section{Hyperbolic geometry}

In this section we show that in the hyperbolic space $H^2(\RRR)$ all
the Euclid's postulates are fulfilled, apart from the fifth which is false.
These properties made hyperbolic geometry iconic.\par
Although the fifth postulate of Euclidean geometry matters only in the 2nd
dimension, our considerations will take place in all hyperbolic spaces.\par
The first two Euclid's postulates deal with straight lines and their segments.
Recall that a \textit{straight line} in a metric space is an isometric image of
the real line, and a \textit{straight line segment} is an isometric image of
the compact interval in $\RRR$. Additionally, for simplicity, we call three
points $a,b,c$ in a metric space $(X,d)$ \textit{metrically collinear} if there
are $x,y,z$ such that $d(x,z) = d(x,y)+d(y,z)$ and the sets $\{x,y,z\}$ and
$\{a,b,c\}$ coincide.\par
To avoid confusions, straight lines in $(\RRR^n,d_e)$ (that is, one-dimensional
affine subspaces) will be called \textsl{Euclidean} lines, whereas straight
lines in $(H^n(\RRR),d_h)$ will be called \textsl{hyperbolic} lines. We will
use analogous naming for other geometric notions---e.g., we will speak about
Euclidean and hyperbolic spheres (as sets of all points equidistinct from
a given point in Euclidean, resp. hyperbolic, metric).\par
The first Euclid's postulate says that any two points of the space can be joint
by a straight line segment. In the modern terminology, it is equivalent for
the metric to be geodesic (that is, convex). The second postulate is about
extending straight line segments to line segments. Both these axioms in
the hyperbolic spaces are fulfilled in a very strict form, as shown by

\begin{thm}{line}
Let $a$ and $b$ be two distinct points of $H^n(\RRR)$.
\begin{enumerate}[\upshape(A)]
\item The metric segment $I(a,b) = \{x \in H^n(\RRR)\dd\ d_h(a,b) =
 d_h(a,x)+d_h(x,b)\}$ is a unique straight line segment in $H^n(\RRR)$ that
 joins $a$ and $b$.
\item The set $L(a,b)$ of all $x \in H^n(\RRR)$ such that $x,a,b$ are metrically
 collinear is a unique hyperbolic line passing through $a$ and $b$.
\item Every isometric map $\gamma\dd \RRR \to H^n(\RRR)$ that sends $0$ to $a$
 is of the form $\gamma(t) = T_a(\sinh(t) z)$ where $z \in H^n(\RRR)$ is such
 that $\|z\| = 1$.
\end{enumerate}
\end{thm}
\begin{proof}
First assume $a = 0$. All the assertions of the theorem in this case will be
shown in a few steps.\par
For any $z \in H^n(\RRR)$ with $\|z\| = 1$ denote by $\gamma_z\dd \RRR \to
H^n(\RRR)$ a map given by $\gamma_z(t) = \sinh(t)z$. This map is isometric,
which can be shown by a straightforward calculation:
\begin{multline*}
d_h(\gamma(s),\gamma(t)) = \cosh^{-1}([\sinh(s)z]\,[\sinh(t)z]-\sinh(s)\sinh(t))
\\=\cosh^{-1}\bigl(\sqrt{(\sinh(s)^2+1)(\sinh(t)^2+1)}-\sinh(s)\sinh(t)\bigr)
\\= \cosh^{-1}(\cosh(s)\cosh(t)-\sinh(s)\sinh(t))=\cosh^{-1}(\cosh(s-t))=|s-t|.
\end{multline*}
Observe that the image of $\gamma_z$ coincides with the linear span of
the vector $z$. Thus, every Euclidean line passing through the zero vector is
also a hyperbolic line.\par
Now let $x,y,z \in H^n(\RRR)$ satisfy $d_h(x,z) = d_h(x,y)+d_h(y,z)$ and let $0
\in \{x,y,z\}$. We claim that $x,y,z$ lie on a Euclidean line. Indeed, if
$y = 0$, it suffices to apply \LEM{tri_ineq}; and otherwise we may assume,
without loss of generality, that $x = 0$. In that case we proceed as follows.
The linear span $L$ of $y$ is a hyperbolic line (by the previous paragraph). So,
there exists an isometry $q\dd L \to L$ that sends $y$ to $0$. By absolute
homogeneity established in \THM{homogeneous}, there exists an isometry $Q\dd
H^n(\RRR) \to H^n(\RRR)$ that extends $q$. So, $Q(L) = L$, $Q(y) = 0$ and
$d_h(Q(x),Q(z)) = d_h(Q(x),Q(y))+d_h(Q(y),Q(z))$. Again, \LEM{tri_ineq} implies
that $Q(z) = t Q(x)$ for some $t \leqsl 0$ (as $Q(x) \neq 0 = Q(y)$). But $Q(x)
= Q(0) \in Q(L) = L$ and therefore also $Q(z) \in L$. So, $x, y, z \in L$ and we
are done.\par
Now we prove items (A)--(C) in the case $a = 0$. Let $z \in H^n(\RRR)$ be
a vector such that $\|z\| = 1$ and $b \in \gamma_z(\RRR)$. It follows from
the last paragraph that each element $x \in H^n(\RRR)$ such that $x,0,b$ are
metrically collinear belongs to $\gamma_z(\RRR)$. We also know that
$\gamma_z(\RRR)$ is a hyperbolic line. This shows that $L(0,b) =
\gamma_z(\RRR)$ and proves (B), from which (A) easily follows. Finally, if
$\gamma\dd \RRR \to H^n(\RRR)$ is an isometric map such that $\gamma(0) = 0$,
then for any $t \in \RRR$, the points $0$, $\gamma(1)$ and $\gamma(t)$ are
metrically collinear, so $\gamma(t) \in L(0,\gamma(1))$. It follows from
the above argument that $L(0,\gamma(1)) = \gamma_z(\RRR)$ for some unit vector
$z \in H^n(\RRR)$. Then $v = \gamma_z^{-1} \circ \gamma\dd \RRR \to \RRR$ is an
isometric map sending $0$ to $0$. Thus $v(t) = t$ or $v(t) = -t$. In the former
case we get $\gamma = \gamma_z$, whereas in the latter we have $\gamma =
\gamma_{-z}$, which finishes the proof of (C).\par
Now we consider a general case. When $a$ is arbitrary, $b \neq a$ and $\gamma\dd
\RRR \to H^n(\RRR)$ is isometric and sends $0$ to $a$, it is easy to verify that
$T_{-a}(I(a,b)) = I(0,T_{-a}(b))$, $T_{-a}(L(a,b)) = L(0,T_{-a}(b))$ and $T_{-a}
\circ \gamma$ is an isometric map from $\RRR$ into $H^n(\RRR)$ that sends $0$ to
$0$. So, the first part of the proof implies that $I(a,b)$ and $L(a,b)$ are
a unique straight line segment joining $a$ and $b$,
and---respectively---a unique hyperbolic line passing through $a$ and $b$.
Similarly, $T_{-a} \circ \gamma = \gamma_z$ for some unit vector $z$. Then
$\gamma = T_a \circ \gamma_z$ and we are done.
\end{proof}

It follows from the above result that two hyperbolic lines either are disjoint
or have a single common point, or coincide.\par
As a consequence of \THM{line}, we obtain the following result, which is at
least surprising when one compares the formulas for $d_h$ and $d_e$.

\begin{cor}{n=1}
The map $(\RRR,d_e) \ni t \mapsto \sinh(t) \in (H^1(\RRR),d_h)$ is an isometry.
\end{cor}

The above result implies, in particular, that $H^1(\RRR)$ admits many
non-isometric dilations. Its assertion is the reason for considering (in almost
whole existing literature) only hyperbolic spaces of dimension greater than
one.\par
Another immediate consequence of \THM{line} reads as follows.

\begin{cor}{geo}
Hyperbolic spaces are geodesic.
\end{cor}

We return to Euclid's postulates. The third of them is about the existence of
circles, which in the metric approach reduces to the statement that for any
center $a$ and a radius $r > 0$ the set, called a \textsl{sphere}, of points
whose distance from $a$ equals $r$ is non-empty (one can require more---e.g.,
that each sphere disconnects the space). Hyperbolic spaces satisfy the third
Euclid's postulate in a very strict way: any hyperbolic sphere around $0$ of
radius $r > 0$ coincides with the Euclidean sphere around $0$ of radius
$\sinh(r)$ (which easily follows from the formulas for $d_h$ and $d_e$);
and---by the absolute homogeneity of $H^n(\RRR)$---any other sphere is the image
of a sphere around $0$ by a global (hyperbolic) isometry. So, all of them
disconnect $H^n(\RRR)$, are pairwise homeomorphic, have homeomorphic complements
etc.\par
The fourth axiom of Euclidean geometry says that all right angles are congruent
(i.e., any of them is the image of any other by a global isometry). It deals
with angles and as such is most difficult (among all Euclid's postulates) to be
adapted in the realm of metric spaces. Knowing that hyperbolic spaces satisfy
first two Euclid's postulates, the following seems to be most intuitive
definition of angles and related notions (cf. \cite{bi2}).

\begin{dfn}{angle}
A (\textit{hyperbolic}) \textit{angle} is an ordered triple $(R_1,a,R_2)$ where
$R_1$ and $R_2$ are closed hyperbolic half-lines in $H^n(\RRR)$ issuing from
$a$. The point $a$ is called the \textit{vertex} of the angle $(R_1,a,R_2)$.\par
Two angles $(R_1,a,R_2)$ and $(M_1,b,M_2)$ are \textit{congruent} if there
exists an isometry $u\dd R_1 \cup R_2 \to M_1 \cup M_2$ such that $u(R_1) =
M_1$, $u(R_2) = M_2$ and $u(a) = b$.\par
Whenever $(R_1,a,R_2)$ is a hyperbolic angle, the half-lines $R_1$ and $R_2$ can
uniquely be enlarged to hyperbolic lines $L_1$ and $L_2$, respectively. Denote
by $R_1'$ and $R_2'$ the half-lines $(L_1 \setminus R_1) \cup \{a\}$ and $(L_2
\setminus R_2) \cup \{a\}$, respectively. We call the angle $(R_1,a,R_2)$
\textit{right} if the angles $(R_1,a,R_2)$, $(R_2',a,R_1)$, $(R_2,a,R_1')$ and
$(R_1',a,R_2')$ are pairwise congruent.
\end{dfn}

Fourth Euclid's postulate in hyperbolic spaces reads as follows.

\begin{pro}{4}
Let $n > 1$.
\begin{enumerate}[\upshape(a)]
\item Among angles with vertex at $0$, hyperbolic right angles coincide with
 Euclidean right angles.
\item For any two hyperbolic right angles $(R_1,a,R_2)$ and $(M_1,b,M_2)$ in
 $H^n(\RRR)$ there exists a global isometry $u\dd H^n(\RRR) \to H^n(\RRR)$ such
 that $u(R_1) = M_1$, $u(R_2) = M_2$ and $u(a) = b$.
\end{enumerate}
\end{pro}
\begin{proof}
First of all, note that in both hyperbolic and Euclidean spaces, congruency of
two angles can be witnessed by a global isometry (since all these spaces are
absolutely homogeneous). Having this in mind, both the claims follow from
the following three facts (first two of which have already been established,
whereas the last is classical):
\begin{itemize}
\item among straight lines passing through $0$, hyperbolic lines coincide with
 Euclidean lines;
\item among maps that leave $0$ fixed, hyperbolic isometries coincide with
 Euclidean isometries;
\item Euclidean right angles are congruent (with respect to $d_e$).
\end{itemize}
\end{proof}

Finally, we arrived at the fifth Euclid's postulate. Its original statement is
about two lines that intersect a third in a way that the sum of the inner angles
on one side is less than two right angles. As it is about measuring angles, more
convenient for us will be one of its equivalent statements, known as Playfair's
axiom, which says that \textsl{through any point not lying on a given straight
line there passes a unique straight line disjoint from this given one}. As we
announced, this axioms is false in hyperbolic geometry. Below we present
a detailed proof of this fact in its full generality. However, only the case
$n = 2$ is interesting in this matter.

\begin{thm}{5}
Let $n > 1$, $L$ be a hyperbolic line in $H^n(\RRR)$ and $a \notin L$. There are
infinitely many hyperbolic lines that pass through $a$ and are disjoint from
$L$.
\end{thm}
\begin{proof}
Thanks to the metric homogeneity of $H^n(\RRR)$, we may and do assume that
$a = 0$. So, $L$ is a hyperbolic line that does not pass through $0$. We claim
that then there exist two linearly independent vectors $a$ and $b$ such that
\begin{equation}\label{eqn:L}
L = \{\sinh(t)a + \cosh(t)b\dd\ t \in \RRR\}.
\end{equation}
Indeed, we know that $L = \gamma(\RRR)$ where $\gamma = T_y \circ \gamma_z$,
$y \neq 0$, $\|z\| = 1$ and $\gamma_z$ is given by $\gamma_z(t) = \sinh(t) z$.
Then $\gamma(t) = \sinh(t)z+([\sinh(t)z]+\frac{\sinh(t)\scalar{z}{y}}{[y]+1})y
= \sinh(t)(z+\frac{\scalar{z}{y}}{[y]+1}y) + \cosh(t)y$. So, substituting $a =
z+\frac{\scalar{z}{y}}{[y]+1}y$ and $b = y$, it remains to check that $a$ and
$b$ are linearly independent to get \eqref{eqn:L}. If $a$ and $b$ were not such,
then $L$ would be a subset of the linear span $Y$ of $y$. But $Y$ is
a hyperbolic line and therefore we would obtain that $L = Y$ and hence $0 \in
L$. This proves \eqref{eqn:L}.\par
Now let $\mu \in \RRR$ be such that $|\mu| > 1$. Put $c = \mu a + b$ and let
$L_{\mu}$ be the linear span of $c$. We claim that $L_{\mu}$ is disjoint from
$L$. Indeed, let $s$ and $t$ be real and assume, on the contrary, that $sc =
\sinh(t)a+\cosh(t)b$ (cf. \eqref{eqn:L}). It follows from the linear
independence of $a$ and $b$ that $\sinh(t) = s\mu$ and $\cosh(t) = s$.
Consequently, $|\sinh(t)| = |\cosh(t) \mu| > |\cosh(t)|$, which is
impossible.\par
So, for any real $\mu$ with $|\mu| > 1$ the set $L_{\mu}$ is a hyperbolic line
disjoint from $L$ and passing through $0$. Since $a$ and $b$ are linearly
independent, these sets $L_{\mu}$ are all different.
\end{proof}

The reader interested in establishing further geometric properties by means of
the metric is referred to \cite{bi2}.

\section{Absolute vs. 3-point homogeneity}

Looking at a complicated formula for the hyperbolic metric it is a natural
temptation to search for a `simpler' realisation of non-Euclidean (hyperbolic)
geometry. Even if it is possible, a new model will not be as `perfect' as
the hyperbolic space described in this paper: its metric will not be geodesic or
else this space will not be $3$-point homogeneous. (Recall that, for a positive
integer $n$, a metric space $(X,d)$ is metrically \textit{$n$-point homogeneous}
if any isometric map $u\dd (A,d) \to (X,d)$ defined on a subset $A$ of $X$ that
has at most $n$ elements is extendable to an isometry $v\dd (X,d) \to (X,d)$.)
The above statement is a consequence of deep achievements (which we neither
discuss here in full details nor give their proofs) of the 50's of the 20th
century due to Wang \cite{wan}, Tits \cite{tit} and Freudenthal \cite{fr1}. They
classified \textsl{all} connected locally compact metric spaces that are
$2$-point homogeneous. All that follows is based on Freudenthal's work
\cite{fr1}; we also strongly recommend his paper \cite{fr2} where main results
of the former article are well discussed (see, e.g., subsection 2.21
therein).\par
To formulate the main result on the classification, \textbf{up to isometry}, of
all connected locally compact $3$-point homogeneous metric spaces, let us
introduce necessary notions.\par
For any positive integer $n$ let $\SSS^n$ stand for the Euclidean unit sphere in
$\RRR^{n+1}$:
\[\SSS^n = \{x \in \RRR^{n+1}\dd\ \|x\| = 1\}.\]
We equip $\SSS^n$ with metric $d_s$ of the \textsl{great-circle distance}, given
by \[d_s(x,y) = \frac{1}{\pi}\arccos(\scalar{x}{y}) = \frac{1}{\pi}
\arccos\frac{2-d_e(x,y)^2}{2}.\]
It is well-known (and easy to prove) that $d_s$ is a metric equivalent to $d_e$
(restricted to $\SSS^n$) such that the metric space $(\SSS^n,d_s)$ has
the following properties:
\begin{itemize}
\item it is geodesic and absolutely homogeneous;
\item it has diameter $1$;
\item any two points of $\SSS^n$ whose $d_s$-distance is smaller than $1$ (that
 is, which are not antipodal) can be joint by a \textbf{unique} straight line
 segment.
\end{itemize}\par
Further, for any subinterval $I$ of $[0,\infty)$ containing $0$ let $\Omega(I)$
be the set of all continuous functions $\omega\dd I \to [0,\infty)$ that vanish
at $0$ and satisfy the following two conditions for all $x, y \in I$:
\begin{enumerate}[($\omega$1)]
\item $x < y \implies \omega(x) < \omega(y)$;
\item $\omega(x+y) \leqsl \omega(x)+\omega(y)$ provided $x+y \in I$.
\end{enumerate}
Finally, let $\Omega = \Omega_{\infty} = \Omega([0,\infty))$ and $\Omega_1 =
\Omega([0,1])$. For any $\omega \in \Omega$ there exists $\lim_{t\to\infty}
\omega(t) \in [0,\infty]$, which will be denoted by $\omega(\infty)$.\par
After all these preparations, we are ready to formulate the main result of this
section.

\begin{thm}{classification}
Each connected locally compact $3$-point homogeneous metric space having more
than one point is isometric to exactly one metric space $(X,d)$ among all listed
below (everywhere below $n$ denotes a positive integer).
\begin{itemize}
\item $X = \RRR^n$, $d = \omega \circ d_e$ where $\omega \in \Omega$ and
 $\omega(1) = \min(1,\frac12\omega(\infty))$.
\item $X = \SSS^n$, $d = \omega \circ d_s$ where $\omega \in \Omega_1$.
\item $X = H^n(\RRR)$, $d = \omega \circ d_h$ where $n > 1$ and $\omega \in
 \Omega$.
\end{itemize}
In particular:
\begin{itemize}
\item all connected locally compact $3$-point homogeneous metric spaces are
 absolutely homogeneous, and each of them is homeomorphic either to a Euclidean
 space or to a Euclidean sphere (unless it has at most one point);
\item each locally compact geodesic $3$-point homogeneous metric space having
 more than one point is isometric to exactly one of the spaces: $(\RRR^n,d_e)$,
 $(\SSS^n,r d_s)$ (where $r > 0$), $(H^n(\RRR),r d_h)$ (where $n > 1$ and
 $r > 0$).
\end{itemize}
\end{thm}

\begin{rem}{b-b}
Busemann's \cite{bus} and Birkhoff's \cite{bi1,bi2} results have a similar
spirit. However, both of them assume, beside metric convexity, also a sort of
uniqueness of straight line segments joining sufficiently close two points, and
their statements are less general.
\end{rem}

It seems that the assertion of \THM{classification} (in this form) has never
appeared in the literature. Nevertheless, we consider this result as
Freudenthal's theorem---not ours. Such a thinking is justified by the fact that
\THM{classification} easily follows from Freudenthal's theorem, stated below,
proved in \cite{fr1}.\par
Denoting by $P^n(\RRR)$ the $n$-dimensional real projective space (realised as
the quotient space of $\SSS^n$ obtained by gluing antipodal points) equipped
with a geodesic metric $d_p(\bar{u},\bar{v}) = \frac{2}{\pi}
\arccos(|\scalar{u}{v}|)$ (where $u, v \in \SSS^n$ and $\bar{u} = \{u,-u\}$ and
$\bar{v} = \{v,-v\}$ denote their equivalence classes belonging to $P^n(\RRR)$),
we can formulate Freudenthal's \textsl{Hauptsatz~IV} from \cite{fr1} as follows.

\begin{thm}{freu}
Let $(Z,\lambda)$ be a connected locally compact metric space that satisfies
the following condition:
\begin{center}
\fbox{\parbox{110mm}{%
There are two positive reals $\gamma$ and $\gamma'$ such that:
\begin{enumerate}[\upshape(a)]
\item $\gamma < \lambda(v,w)$ for some $v, w \in Z$;
\item $\gamma' < \lambda(a,b)$ for some $a, b \in Z$ for which there is $c \in
 Z$ with $\lambda(a,c) = \lambda(b,c) = \gamma$;
\item any isometric map $u\dd (A,\lambda) \to (Z,\lambda)$ defined on
 an arbitrary subset $A$ of $Z$ of the form:
 \begin{itemize}
 \item[$\bullet$] $A = \{x,y\}$ where $\lambda(x,y) = \gamma$, or
 \item[$\bullet$] $A = \{x,y,z\}$ where $\lambda(x,y) = \lambda(y,z) = \gamma$
  and $\lambda(x,z) = \gamma'$
 \end{itemize}
 is extendable to an isometry $v\dd (Z,\lambda) \to (Z,\lambda)$.
\end{enumerate}%
}}
\end{center}
Then there is a unique space $(M,\varrho)$ among $(\RRR^n,d_e)\ (n > 0)$,
$(\SSS^n,d_s)\ (n > 0)$, $(H^n(\RRR),d_h)\ (n > 1)$, $(P^n(\RRR),d_p)\ (n > 1)$
and a homeomorphism $h\dd Z \to M$ such that:
\begin{equation}\label{eqn:iso}
\{h \circ u \circ h^{-1}\dd\ u \in \Iso(Z,\lambda)\} = \Iso(M,\varrho).
\end{equation}
\end{thm}

We prove \THM{classification} in few steps, each formulated as a separate
lemma. All of them are already known but---for the reader's convenience---we
give their short proofs.

\begin{lem}{proj}
For $n > 1$ the space $(P^n(\RRR),d_p)$ is $2$-point, but not $3$-point,
homogeneous.
\end{lem}
\begin{proof}
We use here the notation introduced in the paragraph preceding the formulation
of \THM{freu}.\par
Let $(u,v)$ and $(x,y)$ be two pairs of points of $\SSS^n$ such that
$d_p(\bar{u},\bar{v}) = d_p(\bar{x},\bar{y})$. This means that there is $\epsi
\in \{1,-1\}$ such that $d_s(u,v) = d_s(x,\epsi y)$. Consequently, there exists
$A \in O_{n+1}$ (which is automatically an isometry with respect to $d_s$) such
that $A(u) = x$ and $A(v) = \epsi y$. Every member $B$ of $O_{n+1}$ naturally
induces an isometry $\bar{B}\dd P^n(\RRR) \to P^n(\RRR)$ that satisfies
$\bar{B}(\bar{z}) = \overline{B(z)}$ for any $z \in \SSS^n$. Moreover,
the assignment
\begin{equation}\label{eqn:isoPn}
O_{n+1} \ni B \mapsto \bar{B} \in \Iso(P^n(\RRR),d_p)
\end{equation}
is surjective (this is classical, but non-trivial). So, $\bar{A} \in
\Iso(P^n(\RRR),d_p)$ satisfies $\bar{A}(\bar{u}) = \bar{x}$ and
$\bar{A}(\bar{v}) = \bar{y}$, which shows that $P^n(\RRR)$ is $2$-point
homogeneous.\par
To convince oneself that $P^n(\RRR)$ is not $3$-point homogeneous for $n > 1$,
it is enough to consider two triples $(x,y,z_1)$ and $(x,y,z_2)$ of points of
$\SSS^n$ with $x = (1,0,0,\vec{0}\,)$, $y =
(\frac{\sqrt{2}}{2},\frac{\sqrt{2}}{2},0,\vec{0}\,)$, $z_1 =
(\frac14,\frac14,\frac{\sqrt{14}}{4},\vec{0}\,)$ and $z_2 =
(\frac14,-\frac34,\frac{\sqrt{6}}{4},\vec{0}\,)$ where $\vec{0}$ denotes
the zero vector in $\RRR^{n-2}$ (which has to be omitted when $n = 2$). Observe
that $\scalar{x}{z_1} = \scalar{x}{z_2}$ and $\scalar{y}{z_1} =
-\scalar{y}{z_2}$. Consequently, $d_p(\bar{w},\bar{z}_1) =
d_p(\bar{w},\bar{z}_2)$ for any $w \in \{x,y\}$. If $P^n(\RRR)$ was $3$-point
homogeneous, there would exist an isometry of $P^n(\RRR)$ that fixes $x$ and $y$
and sends $z_1$ onto $z_2$, which is impossible. To see this, assume---on
the contrary---there exists such an isometry. It then follows from
the surjectivity of \eqref{eqn:isoPn} that there exists $B \in O_{n+1}$ such
that $B(x) = \pm x$, $B(y) = \pm y$ and $B(z_1) = \pm z_2$. Then, replacing if
needed $B$ by $-B$, we may and do assume $B(x) = x$ and thus, since
$\scalar{B(x)}{B(u)} = \scalar{x}{u}$ for any $u \in \RRR^{n+1}$, also $B(y) =
y$ and $B(z_1) = z_2$. But then $\scalar{y}{z_2} = \scalar{B(y)}{B(z_1)} =
\scalar{y}{z_1}$ which is false.
\end{proof}

\begin{lem}{omega}
Let $(M,\varrho)$ be a $2$-point homogeneous geodesic metric space having more
than one point and $I$ denote $\varrho(M \times M)$, and let $f\dd I \to
[0,\infty)$ be a one-to-one function such that $\varrho_f = f \circ \varrho$ is
a metric on $M$ equivalent to $\varrho$. Then:
\begin{enumerate}[\upshape(v1)]\addtocounter{enumi}{-1}
\item $I$ is an interval and $f \in \Omega(I)$;
\item $(M,\varrho_f)$ is $2$-point homogeneous;
\item $\Iso(M,\varrho_f) = \Iso(M,\varrho)$;
\item $(M,\varrho_f)$ is absolutely homogeneous iff so is $(M,\varrho)$;
\item $(M,\varrho_f)$ is geodesic iff $f(t) = c t$ for some positive constant
 $c$ (and all $t \in I$).
\end{enumerate}
\end{lem}
\begin{proof}
Since $f$ is one-to-one, any map $u\dd A \to M$ (where $A \subset M$) is
isometric with respect to $\varrho_f$ if and only if it is isometric with
respect to $\varrho$. This implies (v1), (v2) and (v3). It follows from
$1$-point homogeneity of $(M,\varrho)$ that $I = \{\varrho(a,x)\dd\ x \in M\}$
where $a$ is arbitrarily fixed element of $M$. But this, combined with
continuity of $\varrho_f$ with respect to $\varrho$, yields that $f$ is
continuous at $0$. Further, since $M$ is geodesic, $I$ is an interval and for
any $x, y \in I$ with $x+y \in I$ there are points $a, b, c \in M$ such that
$\varrho(a,c) = x+y$, $\varrho(a,b) = x$ and $\varrho(b,c) = y$. Then $f(x+y) =
\varrho_f(a,c) \leqsl \varrho_f(a,b)+\varrho_f(b,c) = f(x)+f(y)$. So,
($\omega$2) holds. This implies that $|\omega(x)-\omega(y)| \leqsl
\omega(|x-y|)$ for any $x, y \in I$. Thus $f$, being continuous at $0$, is
continuous (at each point of $I$). Finally, a continuous one-to-one function
vanishing at $0$ satisfies ($\omega$1) and therefore $f \in \Omega(I)$.\par
It remains to show that if $\varrho_f$ is geodesic, then $f$ is linear. Since
$f \in \Omega(I)$, it is a bijection between $I$ and $J = f(I)$, and $J$ is
an interval. Moreover, $\varrho = f^{-1} \circ \varrho_f$. So, assuming that
$\varrho_f$ is geodesic, it follows from the first part of the proof that also
$f^{-1} \in \Omega(J)$. But if $f \in \Omega(I)$ and $f^{-1} \in \Omega(J)$,
then $f(x+y) = f(x)+f(y)$ for all $x, y \in I$ with $x+y \in I$. The last
equation implies that $f$ is linear (for $f$ is continuous).
\end{proof}

\begin{lem}{2-point}
Assume $(Z,\lambda)$ and $(M,\varrho)$ are two $2$-point homogeneous metric
spaces having more than one point and satisfying the following three conditions:
\begin{itemize}
\item $(M,\varrho)$ is geodesic;
\item there exists $R \in \{1,\infty\}$ such that $\varrho(M \times M) = \{r \in
 \RRR\dd\ 0 \leqsl r \leqsl R\}$;
\item there exists a homeomorphism $h\dd Z \to M$ such that \eqref{eqn:iso}
 holds.
\end{itemize}
Then there exists a unique $\omega \in \Omega_R$ such that $\lambda = \omega
\circ \varrho \circ (h \times h)$. Moreover, $(Z,\lambda)$ is $3$-point or
absolutely homogeneous iff so is $(M,\varrho)$.
\end{lem}
\begin{proof}
Recall that $h \times h\dd Z \times Z \to M \times M$ is given by
$(h \times h)(x,y) = (h(x),h(y))$. Further, for simplicity, denote $I =
\varrho(M \times M)$. It follows from our assumptions that $I = [0,1]$ (if $R =
1$) or $I = [0,\infty)$ (if $R = \infty$).\par
For any four points $a$, $b$, $x$ and $y$ in an arbitrary $2$-point homogeneous
metric space $(Y,p)$ the following equivalence holds:
\[p(a,b) = p(x,y) \iff \exists \Phi \in \Iso(Y,p)\dd\ \Phi(a) = x \textup{ and }
\Phi(b) = y.\]
The above condition, applied for both $(Z,\lambda)$ and $(M,\varrho)$, combined
with \eqref{eqn:iso} yields that
\begin{equation}\label{eqn:aux1}
\lambda(a,b) = \lambda(x,y) \iff \varrho(h(a),h(b)) = \varrho(h(x),h(y)) \qquad
(a,b,x,y \in Z).
\end{equation}
We infer that there exists a one-to-one function $\omega\dd I \to [0,\infty)$
such that $\lambda = \omega \circ \varrho \circ (h \times h)$. Since $\omega
\circ \varrho = \lambda \times (h^{-1} \times h^{-1})$ is a metric equivalent to
$\varrho$, we conclude from \LEM{omega} that $\omega \in \Omega_R$.
The uniqueness of $\omega$ is trivial.\par
Finally, the remainder of the lemma (about $3$-point or absolute homogeneity)
follows from \eqref{eqn:aux1}. Indeed, this condition implies that a map $u\dd
(A,\lambda) \to (Z,\lambda)$ (where $A \subset Z$) is isometric iff the map
$h \circ u \circ h^{-1}\bigr|_{h(A)}\dd (h(A),\varrho) \to (M,\varrho)$ is
isometric.
\end{proof}

\begin{proof}[Proof of \THM{classification}]
First of all, since each of the spaces
\begin{equation}\label{eqn:classic}
(\RRR^n,d_e)\ (n > 0),\ (\SSS^n,d_s)\ (n > 0),\ (H^n(\RRR),d_h)\ (n > 1)
\end{equation}
is absolutely homogeneous, it follows from \LEM{omega} that any space $(X,d)$
listed in the statement of the theorem is also absolutely homogeneous (and, of
course, connected, locally compact and has more than one point). \LEM{omega}
enables us recognizing those among them that are geodesic.\par
Further, if $(Z,\lambda)$ is a connected locally compact $3$-point homogeneous
metric space having more than one point, we infer from \THM{freu} that there are
a metric space $(M,\varrho)$ and a homeomorphism $h\dd Z \to M$ such that
\eqref{eqn:iso} holds and either $(M,\varrho)$ is listed in \eqref{eqn:classic}
or it is $(P^n(\RRR),d_p)$ with $n > 1$. However, since $(M,\varrho)$ is
$2$-point homogeneous (cf. \LEM{proj}), it follows from \LEM{2-point} that it is
also $3$-point homogeneous (because $(Z,\lambda)$ is so). Thus, \LEM{proj}
implies that $(M,\varrho)$ is listed in \eqref{eqn:classic}, and we conclude
from \LEM{2-point} that there is $\widetilde{\omega} \in \Omega(I)$ (where $I =
\varrho(M \times M)$) for which
\begin{equation}\label{eqn:h}
\lambda = \widetilde{\omega} \circ \varrho \circ (h \times h).
\end{equation}
Now if $(M,\varrho)$ is not a Euclidean space, put $\omega =
\widetilde{\omega}$ and observe that $(X,d) = (M,\omega \circ \varrho)$ is
listed in the statement of the theorem and $h\dd (Z,\lambda) \to (X,d)$ is
an isometry, thanks to \eqref{eqn:h}.\par
In the remaining case---when $(M,\varrho) = (\RRR^n,d_e)$---we proceed as
follows. There is $\alpha > 0$ such that $\widetilde{\omega}(\alpha) =
\min(1,\frac12\widetilde{\omega}(\infty))$. We define $\omega \in \Omega$ by
$\omega(t) = \widetilde{\omega}(\alpha t)$. Observe that then $\omega(\infty) =
\widetilde{\omega}(\infty)$ and hence $\omega(1) =
\min(1,\frac12\omega(\infty))$. So, $(X,d) = (\RRR^n,\omega \circ d_e)$ is
listed in the statement of the theorem. What is more, the map
\[(Z,\lambda) \ni z \mapsto \frac{1}{\alpha} h(z) \in (X,d)\]
is an isometry, again by \eqref{eqn:h} (and the definition of $\omega$).\par
The last thing we need to prove is that all the metric spaces $(X,d)$ listed in
the statement of the theorem are pairwise non-isometric. To this end, let
$(M_1,\varrho_1)$ and $(M_2,\varrho_2)$ be two spaces among listed in
\eqref{eqn:classic}, $\omega_1$ and $\omega_2$ be two functions such that
$(X,d) = (M_j,\omega_j \circ \varrho_j)$ for $j\in \{1,2\}$ is listed in
the statement of the theorem; and let $h\dd (M_1,\omega_1 \circ \varrho_1) \to
(M_2,\omega_2 \circ \varrho_2)$ be an isometry. Then
\[\Iso(M_1,\omega_1 \circ \varrho_1) = \Iso(M_1,\varrho_1) = \{h^{-1} \circ u
\circ h\dd\ u \in \Iso(M_2,\varrho_2)\}.\]
So, we infer from the uniqueness in \THM{freu} that $(M_1,\varrho_1) =
(M_2,\varrho_2)$. To simplify further arguments, we denote $(M,\varrho) =
(M_1,\varrho_1)$. Thus, $h$ is an isometry from $(M,\omega_1 \circ \varrho)$
onto $(M,\omega_2 \circ \varrho)$. We conclude that:
\[\varrho(a,b) = \varrho(x,y) \iff \varrho(h(a),h(b)) = \varrho(h(x),h(y))
\qquad (a,b,x,y \in M)\]
(because both $\omega_1$ and $\omega_2$ are one-to-one). The above condition
implies that there is a one-to-one function $f\dd I \to I$ where $I =
\varrho(M \times M)$ such that $\varrho \circ (h \times h) = f \circ \varrho$.
Since both $\varrho$ and $\varrho \circ (h \times h)$ are geodesic, it follows
from \LEM{omega} that there is a constant $c > 0$ such that $f(t) = ct$. So, $h$
is a dilation on $(M,\varrho)$. Now we consider two cases. First assume that
$(M,\varrho)$ is not a Euclidean space. Then $c = 1$, which for hyperbolic
spaces follows from \THM{dilation} and for spheres is trivial (just compare
diameters). Thus, $h \in \Iso(M,\varrho)$ and therefore $\omega_1 \circ \varrho
= \omega_2 \circ \varrho$. Consequently, $\omega_1 = \omega_2$.\par
Finally, assume $(M,\varrho) = (\RRR^n,d_e)$. We have already known that
$\varrho(h(x),h(y)) = c \varrho(x,y)$ for all $x, y \in \RRR^n$. Simultaneously,
$\omega_2(\varrho(h(x),h(y))) = \omega_1(\varrho(x,y))$ for any $x, y \in
\RRR^n$. Both these equations imply that $\omega_2(ct) = \omega_1(t)$ for any
$t \geqsl 0$. In particular, $\omega_2(\infty) = \omega_1(\infty)$ and thus
$\omega_2(c) = \omega_1(1) = \min(1,\frac12 \omega_1(\infty)) = \omega_2(1)$.
Since $\omega_2$ is one-to-one, we get $c = 1$ and hence $\omega_2 = \omega_1$.
\end{proof}

\end{document}